\documentclass{amsart}
\usepackage[english]{babel}
\usepackage{amstext}
\usepackage{latexsym}
\usepackage{amsmath}
\usepackage{amssymb}
\usepackage{amsfonts}
\usepackage{graphicx}
\usepackage{amssymb}
\usepackage{amsthm}
\usepackage{epsfig}
\usepackage{tikz}
\usetikzlibrary{matrix,arrows,decorations.pathmorphing}
\usepackage{graphicx}
\usepackage{enumitem}
\usepackage{xcolor}
\allowdisplaybreaks
\usepackage[utf8]{inputenc}

\newcommand{\GL}{\operatorname{GL}(d,\mathbb{R})}

\input xy
\xyoption{all}
\newtheorem{defn}{Definition}[section]
\newtheorem{coro}[defn]{Corollary}
\newtheorem{lemma}[defn]{Lemma}
\newtheorem{rem}[defn]{Remark}
\newtheorem{prop}[defn]{Proposition}
\newtheorem{theo}[defn]{Theorem}

\newtheorem{claim}[defn]{Claim}
\newtheorem{ltheorem}{Theorem} 

\begin{document}

\title{Simple Lyapunov spectrum of partially hyperbolic diffeomorphisms}

\author[K. Marin]{Karina Marin}
\address{Departamento de Matem\'atica, Universidade Federal de Minas Gerais (UFMG), Av. Ant\^onio Carlos 6627, 31270-901. Belo Horizonte -- MG, Brasil}
\email{kmarin@mat.ufmg.br}

\author[D. Obata]{Davi Obata}
\address{Brigham Young University,  275 TMCB Brigham Young University Provo, UT 84602 }
\email{davi.obata@mathematics.byu.edu}

\author[M. Poletti]{Mauricio Poletti}
\address{Departamento de Matem\'atica, Universidade Federal do Cear\'a (UFC), Campus do Pici,
Bloco 914, CEP 60455-760. Fortaleza -- CE, Brasil}
\email{mpoletti@mat.ufc.br}

\thanks{
KM was partially supported by FAPEMIG, CAPES-Finance Code 001, Instituto Serrapilheira grant number Serra-R-2211-41879.
DO was partially supported by the National Science Foundation under Grant No.\ DMS-2349380. MP was partially supported by CAPES-Finance Code 001, Instituto Serrapilheira grant number Serra-R-2211-41879 and FUNCAP grant AJC 06/2022.}

\begin{abstract}
We study the simplicity of the Lyapunov spectrum of partially hyperbolic diffeomorphisms. We prove that a class of volume-preserving partially hyperbolic diffeomorphisms is $C^r$-accumulated by $C^2$-open sets with simple spectrum. Also we prove that a class of partially hyperbolic maps has simple spectrum generically for the measures of maximal entropy. In order to prove these results, we give a criterion for simplicity of the Lyapunov spectrum in terms of periodic points homoclinically related to the invariant measure.
\end{abstract}

\maketitle

\section{introduction}
The Lyapunov spectrum of a diffeomorphism, with an ergodic measure, is the set of its Lyapunov exponents. The understanding of this set gives information on the asymptotic behavior of almost every orbit. For example, if all the Lyapunov exponents are non-zero, then there exist hyperbolic periodic points. If all the Lyapunov exponents are negative (or positive), it implies that the measure is supported on a periodic orbit. 

In this paper, we deal with the problem of simplicity of the Lyapunov spectrum. By Oseledets theorem, given a diffeomorphism $f:M\to M$, with an ergodic measure $\mu$ there exist at most $d=\dim M$ different Lyapunov exponents. We say that $(f,\mu)$ has \emph{simple Lyapunov spectrum} if it has $d$ different Lyapunov exponents.
The main question we want to answer is how typically $(f,\mu)$ has simple Lyapunov spectrum.

For example, if $A\in \operatorname{SL}(2,\mathbb{R})$ is a linear map and $\mu$ is the Dirac delta measure on the origin, the Lyapunov exponents of $A$ are logarithms of the eigenvalues of $A$. Hence, simple spectrum is an open property in $\operatorname{SL}(2,\mathbb{R})$, but is not dense. There are open sets with pairs of complex eigenvalues of modulus $1$.

A simpler model to study Lyapunov exponents are linear cocycles, that is: when we have maps $f\colon M\to M$, $A\colon M\to \GL$ and an ergodic measure $\mu$. In this case the Lyapunov exponents measure the exponential growth rate of $A^n(x)=A(f^{n-1}(x))\cdots A(x)$. If $A=Df$, the cocycle is called the derivative cocycle of $f$ and the Lyapunov exponents of $A$ coincide with the Lyapunov exponents of the diffeomorphism.

When the base map $f$ is fixed and we study the spectrum with respect to $A$, there are some results about its simplicity.
For random products of matrices, it was proved by Guivarc’h–Raugi \cite{GR} and Gol’dsheid–Margulis \cite{GM} that generically the cocycle has simple spectrum. This was extended by Bonatti-Viana \cite{BV2} and Avila-Viana \cite{AV} for H\"older fiber-bunched linear cocycles over hyperbolic maps. Also, Backes-Lima-Poletti-Varandas~\cite{BPV} proved genericity of simple spectrum for a class of linear cocycles over non-uniformly hyperbolic diffeomorphisms.
For cocycles over partially hyperbolic maps Poletti-Viana~\cite{PV} and Bezerra-Poletti~\cite{BP2} give open sets with simple spectrum. 

For the derivative cocycle, the problem is much more complicated, because the base and the cocyle are related. It is not possible to perturb them independently of each other.

One example of application of simplicity of the Lyapunov spectrum was given by Avila-Crovisier-Wilkinson in the proof of Pugh-Shub's conjecture for the $C^1$-topology. That is, stable ergodicity is $C^1$-dense among partially hyperbolic systems.  In their proof, after perturbation, they obtain some affine horseshoes with large entropy and simple spectrum, and they use this to build superblenders (see Theorems B', 3.2 and C in \cite{ACW}).

In this paper, we find open sets with simple Lyapunov spectrum for two classes of partially hyperbolic diffeomorphisms.

Let $f:M\to M$ be a $C^r, r\geq2$, partially hyperbolic diffeomorphism with splitting $E^{uu}\oplus E^c\oplus E^{ss}$. See Section~\ref{section.preli} for the precise definition. 

The first class is formed by volume-preserving partially hyperbolic diffeomorphisms. Avila-Santamaria-Viana studied the positivity of Lyapunov exponents for linear cocycles over these diffeomorphisms \cite{ASV}. Marin~\cite{M}, Liang-Marin-Yang~\cite{LMY} proved that, when the center dimension is $2$, there is a $C^2$-open and $C^r$-dense set of a class of symplectic maps where the volume measure is hyperbolic.  

Define the asymptotic $u$-conformality of $f$ as \begin{equation}\label{cu}c_u(f)=\lim_{n\to \infty}\frac{1}{n}
\log \max_{x\in M}\|Df_x^n|_{ E^{uu}}\|\|(Df_x^n)^{-1}|_{E^{uu}}\|.\end{equation}
Analogously we define $c_s(f)$ changing $E^{uu}$ by $E^{ss}$. Let $c(f)=\max\{c_u(f),c_s(f)\}$.

Denote $m$ the Lebesgue measure. We say that $f$ is a $\chi$-hyperbolic map, or $m$ is $\chi$-hyperbolic, if all the Lyapunov exponents have absolute value larger than $\chi$.

We say that a $\chi$-hyperbolic map $f$ is \emph{$u$-bunched} if the unstable bundle $E^{uu}$ is $\alpha$-H\"older and
$$c_u(f)<\frac{\alpha\chi}{2}.$$
Analogously, we define \emph{$s$-bunched}. We say that $f$ is \emph{$su$-bunched} when it is $s$ and $u$-bunched.

A $C^2$ volume-preserving diffeomorphim $f$ is said to be \emph{stably ergodic} if any $C^2$ volume-preserving diffeomorphism sufficiently $C^1$-close to $f$ is ergodic.

For fixed $r\geq 2$, take $\mathcal{A}$ to be the set formed by stably-ergodic, non-uniformly hyperbolic, volume-preserving $C^r$ partially hyperbolic diffeomorphisms, such that the center Lyapunov exponents vary continuously with respect to $f$. In this setting we have the following result.
\begin{ltheorem}\label{theo-intro}
    Fix $r\geq 2$ and let $f\in \mathcal{A}$ be $su$-bunched. Then, $f$ can be $C^r$-approximated by $C^2$ open sets of ergodic volume-preserving maps with simple Lyapunov spectrum along $E^{uu}$ and $E^{ss}$ for the volume measure. That is, the Lyapunov exponents associated to $Dg\vert_{E^{uu/ss}}$ have multiplicity one.
\end{ltheorem}

Observe that the condition of the center Lyapunov exponents being continuous is satisfied when $\operatorname{dim}E^c=1$. As a consequence we have the following corollary.
\begin{coro}\label{cor-intro}
    Let $f:M\to M$ be a $C^r$, $r\geq2$, volume-preserving partially hyperbolic diffeomorphism with $\operatorname{dim}E^c=1$.  If $\int\log \|Df|E^c\|dm\neq0$,  then there exists $\delta>0$ such that if $c(f)<\delta$, then $f$ is $C^r$-approximated by $C^2$ open sets of ergodic volume-preserving diffeomorphisms with simple Lyapunov spectrum for the volume measure.
\end{coro}

In Corollary~\ref{theorem-main} and \ref{standard}, we give applications of Theorem \ref{theo-intro} for partially hyperbolic maps with $2$-dimensional center bundle. See next section for the precise statement.

The second class we study are partially hyperbolic diffeomorphisms with a gap condition between topological entropy and stable/unstable entropy and the measures considered are the measures of maximal entropy. The hyperbolicity and finiteness of measures of maximal entropy for these diffeomoprhisms was proved, for the case of one dimensional center bundle by Mongez-Pacifico~\cite{MP1,MP2}, and for $2$-dimensional center bundle by Mongez-Pacifico-Poletti~\cite{MPP}. See Theorem~\ref{theorem-mme-simple} for the statement.

As an application, we can find open sets of Derived from Anosov diffeomorphisms on the torus with simple spectrum for every ergodic MME. 
Let $A\in \GL$, $d\geq 4$ be a hyperbolic matrix with eigenvalues $\lambda_1>\lambda_2>0>\lambda_3\geq \cdots \geq \lambda_d$. 
\begin{ltheorem}\label{theo-intro-DA}
   Let $f:\mathbb{T}^d\to \mathbb{T}^d$ be a $C^r$, $r\geq2$, partially hyperbolic diffeomorphism, with dominated splitting $E^u\oplus E^c\oplus E^s$, $\operatorname{dim} E^u=\operatorname{dim} E^c=1$, isotopic to $A$ along a path of partially hyperbolic diffeomorphisms. Then, there exists $\delta>0$ such that if $c(f)<\delta$, $f$ is $C^r$ accumulated by $C^2$ open sets of diffeomorphisms with simple spectrum for every ergodic measure of maximal entropy.
\end{ltheorem}

The main theorems are proved using a new criterion for simplicity of the Lyapunov spectrum of linear cocycles 
over diffeomorphisms.  Recall that the results of  \cite{GR,GM} are in the random setting, where matrices are chosen randomly with respect to some probability measure on some group of matrices.  In order to prove the simplicity of the spectrum in this setting, they establish a criterion where the group generated by the matrices on the support of the driving measure is irreducible and has a contraction property. 

In \cite{BV2}, they prove a criterion for simplicity of the spectrum of linear cocycles with holonomies over finite Markov shifts when the invariant measure has a continuous local product structure. This condition has two parts similar to the ones for random product. The contraction property is replaced by \emph{pinching}, this means that it exists a periodic point in the support of the invariant measure such that the matrix has simple spectrum. The irreducibility is replaced by \emph{twisting}, that is the periodic pinching point has a homoclinic intersection such that the composition of the holonomies ``mix" the eigenspaces.

The pinching and twisting conditions were extended when the base map is a shift of countably many symbols in \cite{AV} and this criterion was used by Avila-Viana~\cite{AV1} to prove the Zorich-Kontsevich conjecture which states that the Teichm\"uller flow on the moduli space of Abelian differentials on compact Riemann surfaces has simple Lyapunov spectrum. 

In our criterion, Theorem~\ref{theorem-criterium}, the base dynamic is not a shift, it is a non-uniformly hyperbolic diffeomorphism $f$ and the invariant measure $\mu$ is an ergodic equilibrium state. We have similar \emph{pinching} and \emph{twisting} conditions, on a hyperbolic periodic point that is homoclinically related to $\mu$.
The main novelty of this criterion is that it is not in terms of a fixed base dynamics, this allows us to do local perturbations on periodic points and homoclinic intersections.

In order to prove this result, we use \cite{AV} criterion of simplicity over countable Markov shifts with Sarig~\cite{S}, Ovadia~\cite{O} symbolic dynamics and Buzzi-Crovisier-Sarig~\cite{BCS1} irreducibility on homoclinic classes. 

In the next section, we introduce some preliminaries and give the precise statement of the criterion.

\section{Preliminaries and Statements}\label{section.preli}

If $f\colon M\to M$ is a $C^1$ diffeomorphism of a compact manifold $M$ and $\mu$ is an $f$-invariant measure, then, by  Oseledets Theorem,  for $\mu$-almost every point $x\in M$, there exist $k(x)\in \mathbb{N}$, real numbers $\lambda_1(f,x)> \cdots > \lambda_{k(x)}(f,x)$ and a splitting $T_{x}M=E^{1}_x\oplus \cdots \oplus E^{k(x)}_x$ of the tangent bundle at $x$, all depending measurably on the point, such that
$$\lim\limits_{n\rightarrow \pm \infty} \frac{1}{n} \text{log} \left\|Df^{n}_x(v)\right\|= \lambda_j(f,x) \quad \text{for all} \; v\in E^j_x \setminus \{0\}. $$
The real numbers $\lambda_j(f,x)$ are the \emph{Lyapunov exponents} of $f$ at the point $x$. If $(f, \mu)$ is ergodic, then the functions $k(x)$ and $\lambda_j(f,x)$ are constants $\mu$-almost everywhere. 

\begin{defn} We say that a map $f$ has \emph{simple Lyapunov spectrum} for $\mu$ if every Lyapunov exponents has multiplicity one. That is, for $\mu$-almost every point, $\dim\, E^j_x=1$ for every $1\leq j\leq k(x)$. In particular, $k(x)=\dim\, M$.
\end{defn}

A $C^1$ diffeomorphism $f\colon M\to M$ is \textit{partially hyperbolic} if there exist a nontrivial splitting of the tangent bundle $$TM=E^{uu}\oplus E^{c}\oplus E^{ss}$$ invariant under the derivative map $Df$, a Riemannian metric $\left\| \cdot \right\|$ on $M$, and positive continuous functions $\nu$, $\widehat{\nu}$, $\gamma$, $\widehat{\gamma}$ with 
$$\nu < 1 < \widehat{\nu}^{-1} \quad  \text{and} \quad \nu< \gamma < \widehat{\gamma}^{-1}< \widehat{\nu}^{-1},$$ such that for any unit vector $v\in T_{p}M$, 
\begin{equation}\label{ph}
\begin{aligned}
&\left\| Df_{p}(v) \right\|< \nu(p) \quad \quad \text{if} \; v\in E^{ss}(p), \\
\gamma(p) < &\left\| Df_{p}(v) \right\|< \widehat{\gamma}(p)^{-1}\quad  \text{if} \; v\in E^{c}(p),\\
\widehat{\nu}(p)^{-1}< &\left\| Df_{p}(v) \right\| \quad \quad \quad \quad \quad \: \text{if} \; v\in E^{uu}(p).
\end{aligned}
\end{equation}

Partial hyperbolicity is a $C^1$ open condition, that is, any diffeomorphism sufficiently $C^1$-close to a partially hyperbolic diffeomorphism is itself partially hyperbolic. Moreover, if $f\colon M\to M$ is partially hyperbolic, then the stable and unstable bundles $E^{ss}$ and $E^{uu}$ are uniquely integrable and their integral manifolds form two (continuous) foliations $W^{ss}$ and $W^{uu}$, whose leaves are immersed submanifolds of the same class of differentiability as $f$. 

For more information about partially hyperbolic diffeomorphisms we refer the reader to \cite{BDV,HPS,SH}.

\subsection{Volume-preserving maps}

In the following, we present a consequence of Theorem~\ref{theo-intro} about simplicity of the Lyapunov spectrum for the case of $2$-dimensional center bundle. 

\begin{coro}\label{theorem-main}
    Fix $r\geq 2$. There exists a non-empty open set $\mathcal{V}$ of $C^r$ volume-preserving partially hyperbolic diffeomorphisms with $2$-dimensional center bundle, such that $f\in \mathcal{V}$ is non-uniformly hyperbolic and if $f$ is $su$-bunched, then it can be $C^r$ approximated by $C^2$ open sets of ergodic volume-preserving diffeomorphisms with simple Lyapunov spectrum for the volume measure.
\end{coro}

Corollary \ref{theorem-main} is a consequence of Theorem \ref{theo-intro} combined with the results in  \cite{LMY, M}. A particular example where these arguments can be applied is the following. 

Let $f:\mathbb{T}^{2d} \longrightarrow \mathbb{T}^{2d}$ be a Anosov symplectic $C^r$ diffeomorphism such that $c(f)=0$ and $g_{\lambda}:\mathbb{T}^2 \longrightarrow \mathbb{T}^2 $ denote the \textit{standard map} on the $2$-torus defined by, $$g_{\lambda}(z,w)=(z+w,w+ \lambda \sin (2\pi(z+w))).$$ 
\begin{coro}\label{standard} If $\lambda$ is close enough to zero, then $f\times g_{\lambda}$ can be $C^r$-approximated by $C^2$ open sets of ergodic volume-preserving diffeomorphisms with simple Lyapunov spectrum for the volume measure.
\end{coro}

\subsection{Measures of Maximal Entropy}

The symplicity of the Lyapunov spectrum of partially hyperbolic systems also holds when we study measures of maximal entropy. We consider two cases: the first, when the center dimension is equal to 1 and the other when the center dimension is 2. 

Let $r\geq 2$ and $f:M\to M$ be a $C^{r}$ partially hyperbolic diffeomorphism. Denote $h_{top}(f)$ the topological entropy of $f$ and $h^u(f)$ the topological unstable entropy of $f$, see \cite{MP1} for the definition. 

A measure of maximal entropy is an $f$-invariant measure $\mu$ which satisfies $h_\mu(f)=h_{top}(f)$. Here $h_\mu(f)$ denotes the metric entropy of $f$ with respect to $\mu$.

Suppose $\dim E^c=1$, Mongez and Pacifico \cite{MP1} proved that if $h_{top}(f)>h^u(f)$, then $f$ has a finite number of ergodic measures of maximal entropy. Let $\mathcal{V}_1$ be the set formed by $C^r$ partially hyperbolic diffeomorphisms satisfying the above condition and such that for $*=ss,uu$ and $\mathcal{F}_*=Df\vert_{E^{*}}$, 
\begin{equation}\label{eq.condition1}
\|\mathcal{F}_*\| \|\mathcal{F}_*^{-1}\|e^{\frac{\alpha[{h_{top}(f)-h^u(f)]}}{2}}<1,
\end{equation}
where $\alpha>0$ is the H\"older exponent of $x\mapsto E^*_x$. By Lemma 3.4 of \cite{MP1}, the subset $\mathcal{V}_1$ is $C^2$-open.

 Assume now that $\dim E^c=2$ and the center direction admits a dominated splitting $E^c=E_1\oplus E_2$ with $\dim E_1=\dim E_2=1$. Let $h^s(f)$ denote the stable topological entropy of $f$, defined by $h^s(f)=h^u(f^{-1})$. Mongez, Pacifico and the third author \cite{MPP} proved that if $h_{top}(f)>\max\{h^u(f),h^s(f)\}$, then $f$ also has a finite number of ergodic measures of maximal entropy. Let $\mathcal{V}_2$ be the set formed by $C^r$ partially hyperbolic diffeomorphisms satisfying the above condition and such that for $*=ss,uu$ and $\mathcal{F}_*=Df\vert_{E^{*}}$,  
 \begin{equation}\label{eq.condition2}
 \|\mathcal{F}_*\| \|\mathcal{F}_*^{-1}\|e^{\frac{\alpha [h_{top}(f)-\max\{h^u(f), h^s(f)\}]}{2}}<1,
 \end{equation}
 where $\alpha>0$ is the H\"older exponent of $x\mapsto E^*_x$. As before, $\mathcal{V}_2$ is also $C^2$-open, see the proof of Lemma 5.4 of \cite{MPP}.

Conditions \eqref{eq.condition1} and \eqref{eq.condition2} above are related to the value of $c(f)=\max\{c_u(f),c_s(f)\}$ defined in Equation (\ref{cu}).

\begin{ltheorem}\label{theorem-mme-simple}
    Fix $r\geq 2$. For $i=1,2$, let $\mathcal{V}_i$ be the open set defined above. Then, there exists a $C^2$ open and $C^r$ dense subset $\mathcal{U}_i\subset \mathcal{V}_i$ such that for every $f\in \mathcal{U}_i$, $f$ has simple Lyapunov spectrum for any ergodic measure of maximal entropy.
\end{ltheorem}

\subsection{Simplicity criterion} In order to conclude the results on the previous sections, we prove a criterion for the simplicity of the Lyapunov spectrum for linear cocycles over non-uniformly hyperbolic maps.

Let $f:M\to M$ be a $C^{2}$ diffeomorphism. We say that an $f$-invariant measure $\mu$ is \textit{hyperbolic} if the set of points with non-zero Lyapunov exponents has full measure. More precisely, given $\chi>0$, we say that an \emph{ergodic} measure $\mu$ is $\chi$-hyperbolic if $\chi<\min\{|\lambda_{i}(f)|: 1\leq i\leq k\}.$ 

A hyperbolic measure is called of \textit{saddle type} if for $\mu$-almost every $x\in M$, $\lambda_1(f, x) < 0 < \lambda_{k(x)}(f, x)$. In this case, by Pesin's Stable Manifold Theorem \cite{P}, $\mu$-almost every $x$ admit $C^r$ embedded manifolds $W^s_{loc}(x)$ and $W^u_{loc}(x)$ such that, $$\text{for every } y\in W^s_{loc}(x),\; \limsup_{n\to \infty} \frac{1}{n} \log d(f^n(x), f^n(y))<0,$$ and analogously for $W^u_{loc}(x)$ considering backwards iterates. For these points the global stable and unstable manifolds are defined by, $$W^s(x)=\bigcup_{n\geq 0} f^{-n}[W^s_{loc}(f^n(x)] \quad \text{ and } \quad W^u(x)=\bigcup_{n\geq 0} f^{n}[W^u_{loc}(f^{-n}(x)].$$ 

\begin{defn}\label{homoclinic.point}
Let $f$ be a $C^{2}$ diffeomorphism and $\mu$ an ergodic hyperbolic measure of saddle type. We say that a hyperbolic periodic point $p\in M$ is homoclinically related to $\mu$ and denoted $p\sim_h \mu$ if there exists a measurable set $X\subset M$, $\mu(X)>0$ such that $W^u(x)\pitchfork W^s(\mathcal{O}(p))\neq\emptyset$ and $W^s(x)\pitchfork W^u(\mathcal{O}(p))\neq\emptyset$ for any $x\in X$.  
Here $\mathcal{O}(p)$ is the orbit of $p$, and $W^*(\mathcal{O}(p))=\bigcup_{y\in \mathcal{O}(p)}W^*(y)$.
\end{defn}

The main result in this section is going to hold for two types of measures: ergodic equilibrium states of H\"older potentials and ergodic SRB measures. 

Given a continuous map $\varphi\colon M\to \mathbb{R}$, an $f$-invariant measure $\mu$ is an \emph{equilibrium state} for the potential $\varphi$ if $h_{\mu}(f)+\int \varphi d\mu=P(f, \varphi),$ where $P(f, \varphi)=\sup_{\upsilon\in \mathcal{M}_f} \{h_{\upsilon}(f) + \int \varphi d\upsilon\}$ and $\mathcal{M}_f$ denotes the subset of $f$-invariant measures. In particular, measures of maximal entropy are equilibrium states.

On the other hand, an ergodic measures $\mu$ is a \emph{SRB} if it verifies $h_{\mu}(f)=\sum_{\lambda_j(f)>0} \lambda_j(f)\dim E^j.$

Consider $\pi\colon \mathcal{E}\to M$ a vector bundle such that the fibers are isomorphic to $\mathbb{R}^d$ and ${(U, \phi_U)}_{U\in \mathcal{W}}$ a family of local charts for $\mathcal{E}$. For every $U, V \in \mathcal{W}$ such that $U\cap V$ is non-empty the map $(U\cap V)\times \mathbb{R}^d\to (U\cap V)\times \mathbb{R}^d$ is of the form $(x,\xi)\mapsto (x, g_x(\xi))$. Given $\alpha>0$ we say that the vector bundle $\pi\colon \mathcal{E}\to M$ is an ${\alpha}$-H\"older continuous vector bundle if for every family of local charts ${(U, \phi_U)}_{U\in \mathcal{W}}$ and every $U, V \in \mathcal{W}$ with non-empty intersection, the map $x\mapsto g_x$ defined on $U\cap V$ is $\alpha$-H\"older continuous. 

Let $\pi\colon \mathcal{E}\to M$ be an $\alpha$-H\"older continuous vector bundle and $f\colon M\to M$ a $C^2$ diffeomorphism. A \textit{linear cocycle} over $f$ is a continuous transformation $\mathcal{F}\colon \mathcal{E}\to \mathcal{E}$ satisfying $\pi \circ \mathcal{F}=f\circ \pi$ and acting by linear isomorphisms $\mathcal{F}_x\colon \mathcal{E}_x\to \mathcal{E}_{f(x)}$ on the fibers. We say that $\mathcal{F}$ is $\alpha$-H\"older continuous if its expression on local coordinates $(U\cap f^{-1}(V))\times \mathbb{R}^d\to V\times \mathbb{R}^d$, $(x,\xi)\mapsto (f(x), A(x)\xi)$ is such that $x\mapsto A(x)$ is $\alpha$-H\"older continuous. Observe that this condition does not depend on the choice of the local charts since $\pi\colon \mathcal{E}\to M$ is $\alpha$-H\"older continuous. 


Every $\alpha$-H\"older continuous map $A:M\to \GL$ defines an $\alpha$-H\"older continuous linear cocyle $\mathcal{F}_{A}\colon M\times \mathbb{R}^d\to  M\times \mathbb{R}^d$ by $(x,\xi)\mapsto (f(x), A(x)\xi)$. In this case, we identify $\mathcal{F}_{A}$ with $A$ and we call $A$ a linear cocycle.

Given a linear cocycle $\mathcal{F}$ and an $f$-invariant measure $\mu$, we can define the Lyapunov exponents of the cocycle $\mathcal{F}$ by applying Oseledets Theorem as before. That is, for $\mu$-almost every point $x\in M$, there exist $k(x)\in \mathbb{N}$, real numbers $\lambda_1(\mathcal{F},x)> \cdots > \lambda_{k(x)}(\mathcal{F},x)$ and a splitting $\mathbb{R}^d=E^{1}_x\oplus \cdots \oplus E^{k(x)}_x$, all depending measurably on the point, such that
$$\lim\limits_{n\rightarrow \pm \infty} \frac{1}{n} \text{log} \left\|\mathcal{F}^{n}_x(v)\right\|= \lambda_j(\mathcal{F},x) \quad \text{for all} \; v\in E^j_x \setminus \{0\}. $$

\begin{defn} We say that $\mathcal{F}$ has \emph{simple Lyapunov spectrum for $\mu$} if all its Lyapunov exponents has multiplicity one. 
\end{defn}

\begin{defn}\label{fb}
For $\chi>0$, $\mathcal{F}$ is said to be $\chi$-fiber bunched if $\mathcal{F}$ is an $\alpha$-H\"older continuous linear cocycle and there exist $C>0$ and $\lambda\in (0,1)$ such that $$\|\mathcal{F}^n_x\| \|(\mathcal{F}^n_x)^{-1}\|e^{-\chi\alpha|n|}<C\lambda^{|n|},$$ for every $x\in M$ and $n\in \mathbb{Z}$.
\end{defn}

Let $p$ be an $\chi$-hyperbolic $f$-periodic point with period $n(p)$ and $\mathcal{F}$ a $\chi$-fiber bunched cocycle. The fiber-bunched condition implies the existence of the following limits,
\begin{equation}\label{holonomies.diffeo} 
\begin{aligned}
H^s_{y,z}&=\lim_{n\to \infty} (\mathcal{F}_z^{n})^{-1} \circ \mathcal{F}_y^{n} \quad y,z\in W^s(p)\\
H^{u}_{y,z}&=\lim_{n\to \infty} (\mathcal{F}_z^{-n})^{-1}\circ \mathcal{F}_y^{-n}, \quad y,z\in W^u(p).
\end{aligned}
\end{equation}

Observe that the function compositions in the limits above are well-defined, because if $y,z\in W^s(p)$, then for $n$ large enough, $f^n(y)$ and $f^n(z)$ are sufficiently close to identify $\mathcal{E}_{f^n(y)}$ and $\mathcal{E}_{f^n(z)}$ through local charts. Analogous argument  for $y, z\in W^u(p)$. The limits do not depend on the identification that was chosen.

For $z\in W^u_{loc}(p)\pitchfork f^{-l}W^s_{loc}(p)$, the transition map $\psi^{\mathcal{F}}_{p, z}\colon \mathcal{E}_{p}\to \mathcal{E}_{p}$ is defined by \begin{equation}\label{transition.A} \psi^{\mathcal{F}}_{p, z}=H^{s}_{f^{l}(z),p}\circ \mathcal{F}_z^{l}\circ H^u_{p,z}.\end{equation}

\begin{defn}\label{pinch.twist.nuh} Let $\mathcal{F}$ be a $\chi$-fiber bunched cocycle,
\begin{enumerate}
    \item[(P)] A $\chi$-hyperbolic point $p$ is pinching if all eigenvalues of $\mathcal{F}_p^{n(p)}$ have distinct absolute values;
    \item[(T)] A point $z\in W^u_{loc}(p)\pitchfork f^{-l}W^s_{loc}(p)$ is twisting with respect to $p$ if for any invariant subspaces (sums of eigenspaces) $E$ and $F$ of $\mathcal{F}_p^{n(p)}$ with $\text{dim}(E)+\text{dim}(F)=d$, it holds $\psi^{\mathcal{F}}_{p,z}(E)\cap F=\{\emptyset\}.$
\end{enumerate}
\end{defn}

\begin{ltheorem}\label{theorem-criterium}
Let $f:M\to M$ be $C^{2}$ diffeomorphism, $\mathcal{F}:\mathcal{E}\to \mathcal{E}$ be $\chi/2$-fiber bunched linear cocycle and $\mu$ be an ergodic $\chi$-hyperbolic measure of saddle type which is an equilibrium state of a H\"older potential, or an SRB measure. Then, if it exist a $\chi$-hyperbolic pinching periodic point $p\sim_h \mu$ and $z\in W^u(p)\pitchfork W^s(p)$ twisting with respect to $p$ then $\mathcal{F}$ has simple Lyapunov spectrum for $\mu$.
\end{ltheorem}
\begin{rem}
    This theorem can be generalized for maps $f:M\to M$ satisfying \emph{(A1)-(A7)} of \cite{LOP}, taking the pinching and twisting points in the inverse limit. 
    
    The criterion stated in Theorem \ref{theorem-criterium} also holds for $C^{1+\beta}$ maps for every $\beta>0$. Therefore, the theorems and corollaries presented in the previous sections can also be extended to the $C^{1+\beta}$-topology. 

    Moreover, the $C^2$-open neighborhood in the statement of the mentioned results can be replaced by a $C^1$-open neighborhood of $C^2$ diffeomorphism with bounded $C^2$-norm. This condition is necessary to guarantee the continuity of the holonomies as functions of the map. In the case of $C^{1+\beta}$ diffeomorphism, it is enough to ask for a $C^1$ neighborhood with uniform H\"older constant. 
\end{rem}


\section{Simple linear cocycles}

Let $\mathcal{G}$ be a directed graph with a countable vertex set $\mathcal{R}$ such that every vertex has at least one incoming and one outgoing edge. The topological Markov shift associated to $\mathcal{G}$ is the pair $(\hat{\Sigma}, \hat{\sigma})$ where $$\hat{\Sigma}=\{\underline{R}=\{R_i\}_{i\in \mathbb{Z}}\in \mathcal{R}^{\mathbb{Z}} : R_i\rightarrow R_{i+1}, \forall\; i \in \mathbb{Z}\},$$ and $\hat{\sigma}\colon \hat{\Sigma}\to \hat{\Sigma}$ is the left shift, $\hat{\sigma}(\{R_i\}_{i\in \mathbb{Z}})=\{R_{i+1}\}_{i\in \mathbb{Z}}$.

Given $\chi>0$, define a metric on $\hat{\Sigma}$ by \begin{equation}\label{distance}d(\{R_i\}_{i\in \mathbb{Z}}, \{S_i\}_{i\in \mathbb{Z}})=\exp\left(-\frac{\chi}{2}\min{|n|: R_n\neq S_n}\right).\end{equation} With this metric, $\hat{\Sigma}$ is a complete separable metric space and $\hat{\sigma}$ is a hyperbolic homeomorphism.

For every $\underline{R}\in \hat{\Sigma}$ the local stable and unstable sets are given by
\begin{equation*}
\begin{aligned}
W^s_{loc}(\underline{R})=&\{(S_n)_{n\in\mathbb{Z}}\in \hat{\Sigma}: \ S_n=R_n \text{ with } n\geq 0\},\\ 
W^u_{loc}(\underline{R})=&\{(S_n)_{n\in\mathbb{Z}}\in \hat{\Sigma}: \ S_n=R_n \text{ with } n\leq 0\}.
\end{aligned}
\end{equation*}

Observe that if $\underline{S}\in W^s_{loc}(\underline{R})$, then $d(\hat{\sigma}(\underline{R}),\hat{\sigma}(\underline{S}))\leq e^{-\frac{\chi}{2}}d(\underline{R}, \underline{S})$ and if $\underline{S}\in W^u_{loc}(\underline{R})$, then $d(\hat{\sigma}^{-1}(\underline{R}),\hat{\sigma}^{-1}(\underline{S}))\leq e^{-\frac{\chi}{2}}d(\underline{R}, \underline{S})$.

Consider $$\hat{\Sigma}^{+}=\left\{\{R_i\}_{i\geq 0}: \text{ there exists } \{S_i\}_{i\in \mathbb{Z}} \text{ such that } \{R_i\}_{i\geq 0}=\{S_i\}_{i\geq 0}\right\},$$
$$\hat{\Sigma}^{-}=\left\{\{R_i\}_{i\leq 0}: \text{ there exists } \{S_i\}_{i\in \mathbb{Z}} \text{ such that } \{R_i\}_{i\leq 0}=\{S_i\}_{i\leq 0}\right\}.$$

We denote by $P^{+}\colon\hat{\Sigma}\to \hat{\Sigma}^{+}$ and $P^{-}\colon\hat{\Sigma}\to \hat{\Sigma}^{-}$ the projections obtained by dropping the negative and the positive coordinates, respectively, of a sequence in $\hat{\Sigma}$. 

Given a symbol $R\in \mathcal{R}$, define the cylinder $$[R]=\left\{\{S_i\}_{i\in \mathbb{Z}}\in \hat{\Sigma}: S_0=R\right\}.$$

A $\hat{\sigma}$-invariant probability measure $\hat{\mu}$ is said to have local product structure if the normalized restriction of $\hat{\mu}$ to the cylinder $[R]$ has the form $\hat{\mu}\vert [R]=\hat{\rho}\times (\hat{\mu}^{+}\times \hat{\mu}^{-}),$ where $\hat{\mu}^{+}=P^{+}_{*}\hat{\mu}$, $\hat{\mu}^{-}=P^{-}_{*}\hat{\mu}$ and the density $\hat{\rho}$ is measurable. Moreover, we say that $\hat{\mu}$ has \emph{continuous local product structure} if the density $\hat{\rho}\colon \hat{\Sigma}\to \mathbb{R}_{+}$ is uniformly continuous and bounded away from zero and infinity. 

As before, a continuous function $\hat{A}\colon \hat{\Sigma}\to \text{GL}(d, \mathbb{R})$ defines a linear cocycle over $(\hat{\Sigma}, \hat{\sigma})$ and as a consequence of the Theorem of Oseledets, given a $\hat{\sigma}$-invariant measure $\hat{\mu}$, we can define the Lyapunov exponents of the cocycle $\hat{A}$. 

\begin{defn}\label{def.holonomies} Let $\hat{A}$ be a continuous linear cocycle over $(\hat{\Sigma}, \hat{\sigma})$.

A \emph{stable holonomy} for $\hat{A}$ is a collection of linear isomorphisms $H^{s, \hat{A}}_{\underline{R},\underline{S}}\colon \mathbb{R}^d\to \mathbb{R}^d $, defined for every $\underline{R},\underline{S}$ in the same local stable set, which satisfy the following properties,
\vspace{0.2cm}

\begin{enumerate}[label=\emph{(\alph*)}]
\item $H^{s, \hat{A}}_{\underline{U},\underline{S}}\circ H^{s, \hat{A}}_{\underline{R},\underline{U}}= H^{s, \hat{A}}_{\underline{R},\underline{S}}$ and $H^{s, \hat{A}}_{\underline{R},\underline{R}}=Id$;
\vspace{0.2cm}

\item $H^{s, \hat{A}}_{\hat{\sigma}(\underline{R}),\hat{\sigma}(\underline{S})}=\hat{A}({\underline{S}})\circ H^{s, \hat{A}}_{\underline{R},\underline{S}}\circ \hat{A}(\underline{R})^{-1}$; 
\vspace{0.2cm}

\item $(\underline{R},\underline{S})\mapsto H^{s, \hat{A}}_{\underline{R},\underline{S}}$ is continuous.
\end{enumerate}
\vspace{0.2cm}

A \emph{unstable holonomy} for $\hat{A}$ is defined analogously for points in the same local unstable set. We use the expression \emph{invariant holonomies} to refer to both stable and unstable holonomies.
\end{defn}


By \cite{BGV}, an $\alpha$-H\"older linear cocycle $\hat{A}\colon \hat{\Sigma}\to \text{GL}(d, \mathbb{R})$ which is $\frac{\chi}{2}$-bunched (Definition \ref{fb}), always admits invariant holonomies. In this case, \begin{equation}\label{holonomies} 
\begin{aligned}
H^{s, \hat{A}}_{\underline{R},\underline{S}}&=\lim_{n\to \infty} \hat{A}^n(\underline{S})^{-1}\circ \hat{A}^n(\underline{R}), \quad \underline{S}\in W^s_{loc}(\underline{R})\\
H^{u, \hat{A}}_{\underline{R},\underline{S}}&=\lim_{n\to \infty} \hat{A}^{-n}(\underline{S})^{-1}\circ \hat{A}^{-n}(\underline{R}), \quad \underline{S}\in W^u_{loc}(\underline{R}).
\end{aligned}
\end{equation} Moreover, there exists $L>0$ such that for $*\in\{s,u\}$, $$\|H^{*, \hat{A}}_{\underline{R},\underline{S}} - Id\|\leq L d(\underline{R}, \underline{S})^{\alpha}.$$

\subsection{Simplicity Criterion}

Let $\underline{P}$ be a $\hat{\sigma}$-periodic point with period $q\geq 1$. A point $\underline{U}\in W^u_{loc}(\underline{P})$ is called homoclinic if there exists some multiple of $q$, $l\geq 1$, such that $\hat{\sigma}^l\underline{U}\in W^s_{loc}(\underline{P})$. 

Suppose $\hat{A}$ is a linear cocycle admitting invariant holonomies. The transition map $\psi^{\hat{A}}_{\underline{P}, \underline{U}}\colon \mathbb{R}^{d}\to \mathbb{R}^{d}$ is defined by \begin{equation}\label{transition.cocycle} \psi^{\hat{A}}_{\underline{P}, \underline{U}}=H^{s, \hat{A}}_{\hat{\sigma}^{l}(\underline{U}),\underline{P}}\circ \hat{A}(\underline{U})^{l}\circ H^{u, \hat{A}}_{\underline{P},\underline{U}}.\end{equation}

\begin{defn}\label{pinch.twist.} Let $\hat{\mu}$ be a $\hat{\sigma}$-invariant measure. A linear cocycle $\hat{A}$ admitting invariant holonomies is \emph{simple for $\hat{\mu}$} if there exists a $\hat{\sigma}$-periodic point $\underline{P}$ and some homoclinic point $\underline{U}$, both in the support of $\hat{\mu}$, such that:
\begin{enumerate}
    \item all eigenvalues of $\hat{A}(\underline{P})^{q}$ have distinct absolute values;
    \item for any invariant subspaces (sums of eigenspaces) $E$ and $F$ of $\hat{A}(\underline{P})^{q}$ with $\text{dim}(E)+\text{dim}(F)=d$, it holds $\psi^{\hat{A}}_{\underline{P}, \underline{U}}(E)\cap F=\{\emptyset\}.$
\end{enumerate}
\end{defn}
Analogously to Definition \ref{pinch.twist.nuh}, property (1) is called pinching and property (2) is called twisting. 

By Theorem A of \cite{AV} (see also Appendix A.1), the following result holds.

\begin{theo}\label{simple cocycle}
    Let $(\hat{\Sigma}, \hat{\sigma})$ be a topological Markov shift and $\hat{\mu}$ a $\hat{\sigma}$-invariant ergodic measure with continuous local product structure. 
    
If the linear cocycle $\hat{A}$ is simple for $\hat{\mu}$, then $\hat{A}$ has simple Lyapunov spectrum. 
\end{theo}

\section{Proof of Theorem \ref{theorem-criterium}}

The following result is a consequence of several theorems proved in \cite{BCS1, BPV, O, S} and it allow us to relate a diffeomorphism $f$ with the setting of topological Markov shifts explained in the previous section.

\begin{theo}\label{coding} Let $f:M\to M$, be $C^{2}$ diffeomorphism and $\mu$ be an ergodic $\chi$-hyperbolic measure which is an equilibrium state of a H\"older potential, or an SRB measure. Then, there exists a topological Markov shift $(\hat{\Sigma}_f, \hat{\sigma}_f)$ and a H\"older continuous map $\hat{\pi}_f\colon \hat{\Sigma}_f\to M$ such that $\hat{\pi}_f\circ \hat{\sigma}=f\circ \hat{\pi}_f$. Moreover:
\begin{enumerate}
    \item $\hat{\Sigma}_f$ is irreducible.
    \item There exists a $\hat{\sigma}_f$-invariant ergodic measure $\hat{\mu}_f$ with continuous local product structure such that $(\hat{\pi}_f)_{*}\hat{\mu}_f=\mu$.
    \item $\hat{\mu}_f$ has full support. 
\end{enumerate}
    \end{theo}
\begin{proof}
    The existence of the irreducible Markov shift that codifies $\mu$ almost every point is a consequence of \cite{BCS1}. The fact that $\hat{\mu}_f$ has continuous product structure comes from \cite{BPV} and having full support comes from the fact that $\hat{\mu}_f$ is an equilibrium state and \cite[Theorem~1.2]{BS}.
\end{proof}

\begin{rem}\label{lip} By considering the distance in Equation (\ref{distance}), we can assume that the map $\hat{\pi}_f$ is Lipschitz. See Remark 3.2 of \cite{BPV}.
    
\end{rem}

Let $f$ be a $C^2$ diffeomorphism and $\mu_1$ an $\mu_2$ be ergodic hyperbolic measures of saddle type. The notation $\mu_1\preceq \mu_2$ means that there are measurable sets $X_1$ and $X_2$ with $\mu_i(X_i)>0$ such that for all $(x_1, x_2)\in X_1\times X_2$,  $W^u(x_1)\pitchfork W^s(x_2)\neq\emptyset$. 

\begin{defn}
We say that $\mu_1$ and $\mu_2$ are homoclinically related if $\mu_1\preceq \mu_2$ and $\mu_2\preceq \mu_1$. We denote $\mu_1\sim_h \mu_2$.
\end{defn}

\begin{defn} A hyperbolic transitive set $\Lambda\subset M$ is homoclinically related to an ergodic hyperbolic measure of saddle type $\mu$ if for some $\nu$ ergodic measure of $f\vert \Lambda$, we have $\nu\sim_h \mu$. We write $\Lambda\sim_h \mu$.
\end{defn}

Observe that any two ergodic hyperbolic measures of $f\vert \Lambda$ are homoclinically related. In particular, any hyperbolic periodic point $p\in \Lambda$ and the measure supported on its orbit, are homoclinically related to $\Lambda$.

By Theorem 3.14 of \cite{BCS1}, if $f$ and $\mu$ are in the hypotheses of Theorem \ref{coding}, $\mu$ is of saddle type and $\Lambda\sim_h \mu$, then we can assume there exists a transitive invariant compact set $K\subset \hat{\Sigma}_{f}$ such that $\hat{\pi}_{f}(K)=\Lambda$. Here $(\hat{\Sigma}_{f}, \hat{\sigma}_{f})$ is the topological Markov shift given by Theorem \ref{coding}.

\begin{lemma}\label{cocycle-A} Consider $f$, $\mathcal{F}$ and $\mu$ as in the hypotheses of Theorem \ref{theorem-criterium}. Denote $(\hat{\Sigma}_f, \hat{\sigma}_f, \hat{\mu}_f)$ the topological Markov shift given by Theorem \ref{coding} for $f$. There exists a linear cocycle $\hat{A}_{\mathcal{F}}\colon \hat{\Sigma}_f\to \text{GL}(d, \mathbb{R})$ admitting invariant holonomies and such that the Lyapunov exponents of $\mathcal{F}$ for $\mu$ coincides with the Lyapunov exponents of $\hat{A}_{\mathcal{F}}$ for $\hat{\mu}_f$.
\end{lemma}
\begin{proof}
 Refining the Markov representation, if necessary, we can assume that the image of each cylinder $[R]$ of $\hat{\Sigma}_f$ is contained inside of an open set on which the bundle $M\mapsto \mathcal{E}$ is trivializable. Then, we can extend the trivialization and define a continuous map $L\colon \hat{\Sigma}_f \times \mathbb{R}^d \to \mathcal{E}$ such that $L(\underline{R})\colon \mathbb{R}^d \to \mathcal{E}_{\hat{\pi}_f(\underline{R})}$ is a linear isomorphism.

The map $\hat{A}_{\mathcal{F}}\colon \hat{\Sigma}_f\to \text{GL}(d, \mathbb{R})$ defined by \begin{equation}\label{cocycle}
\hat{A}_{\mathcal{F}}(\underline{R})=L(\hat{\sigma}_f((\underline{R})))^{-1} \circ \mathcal{F} \circ L(\hat{\pi}_f((\underline{R})))
\end{equation} is a continuous linear cocycle. Moreover, by Remark \ref{lip}, we know that $\hat{A}_{\mathcal{F}}$ is $\alpha$-H\"older and $\frac{\chi}{2}$-fiber bunched and therefore admits invariant holonomies.

The relation in Equation (\ref{cocycle}) implies that the Lyapunov spectrum of $\hat{A}_{\mathcal{F}}$ and $\mathcal{F}$ coincides. 
\end{proof}

\subsection{Proof of Theorem \ref{theorem-criterium}} Let $f:M\to M$ be $C^{2}$ diffeomorphism, $\mathcal{F}:\mathcal{E}\to \mathcal{E}$ be $\chi/2$-fiber bunched linear cocycle and $\mu$ be an ergodic $\chi$-hyperbolic measure of saddle type which is an equilibrium state of a H\"older potential, or an SRB measure as in the hypotheses of Theorem \ref{theorem-criterium}. 

Applying Theorem \ref{coding} and Lemma \ref{cocycle-A}, we consider the topological Markov shift $(\hat{\Sigma}_f, \hat{\sigma}_f, \hat{\mu}_f)$ and the linear cocycle $\hat{A}_{\mathcal{F}}\colon \hat{\Sigma}_f\to \text{GL}(d, \mathbb{R})$. Observe that in order to prove Theorem \ref{theorem-criterium} it is enough to prove that the Lyapunov spectrum of $\hat{A}_{\mathcal{F}}$ is simple. Then, by Theorem \ref{simple cocycle}, we only need to verify that $\hat{A}_{\mathcal{F}}$ is simple for $\hat{\mu}_f$.

Let $p$ and $z$ be given by the hypotheses of Theorem \ref{theorem-criterium}. That is, $p$ is a $\chi$-hyperbolic pinching periodic point, $p\sim_h \mu$ and $z\in W^u(p)\pitchfork W^s(p)$ is twisting with respect to $p$. Considering the orbit of $p$ and $z$ we can define a transitive hyperbolic set $\Lambda\subset M$ such that $p,z\in \Lambda$. Moreover, since $p\sim_h \mu$, we conclude that $\Lambda\sim_h \mu$ and therefore there exists $K\subset \hat{\Sigma}_{f}$ such that $\hat{\pi}_{f}(K)=\Lambda$.

By the observation above, there exist $\underline{P}, \underline{U}\in K$ such that $\underline{P}$ is a $\hat{\sigma}_{f}$-periodic point and $\underline{U}\in W^u_{loc}(\underline{P})$ is homoclinic with $\hat{\pi}_{f}(\underline{P})=p$ and $\hat{\pi}_{f}(\hat{\sigma}_{f}^l(\underline{U}))=z$ for some $l\geq 1$ multiple of the period of $\underline{P}$. 

Observe that $\hat{A}_{\mathcal{F}}$ satisfies properties (1) and (2) in Definition \ref{pinch.twist.}. In fact, by the definition of the holonomies of $\mathcal{F}$ in Equation (\ref{holonomies.diffeo}), the definition of the holonomies of $\hat{A}_{\mathcal{F}}$ in Equation (\ref{holonomies}) and the relation of the two cocycles in Equation (\ref{cocycle}), we conclude that the transition map $\psi^{\hat{A}_{\mathcal{F}}}_{\underline{P}, \underline{U}}$ defined in Equation (\ref{transition.cocycle}) is conjugated by the linearization $L$ to the map $\psi^{\mathcal{F}}_{p,z}$ defined by Equation (\ref{transition.A}). 

Recall that $\hat{\mu}_{f}$ given by Theorem \ref{coding} has full support. Then, we deduce that $\hat{A}_{\mathcal{F}}$ is a simple cocycle for $\hat{\mu}_{f}$ and applying Theorem \ref{simple cocycle} and Lemma \ref{cocycle-A} we conclude the proof.  \qed

\begin{rem}
    In \cite{BPV}, simplicity of the spectrum was proved to be generic for fiber-bunched linear cocycles over a fix non-uniformly hyperbolic diffeomorphism in the base. In that result it was used the coding of \cite{O,S} to conjugate the base dynamics by a Markov shift and then perturb the cocycle to make it simple and use the Avila-Viana criterion, Theorem~\ref{simple cocycle}. In that argument,  it was essential that the base dynamics was conjugated to a fixed Markov shift. This cannot be done here, because when we perturb $f$ the Markov shift of Theorem \ref{coding} changes in a non trivial way. Because of that, it was essential to have a criterion that does not rely on a fixed Markov shift.
\end{rem}

\section{Proof of Theorem~\ref{theo-intro-DA} and \ref{theorem-mme-simple}} \label{sec.proofthmB}

Fix $r\geq 2$ and for $i=1,2$, consider the subset $\mathcal{V}_i$ defined in Section 2.2. In the following, we prove the simplicity of the Lyapunov spectrum for the measures of maximal entropy in an $C^2$ open and $C^r$ dense subset of $\mathcal{V}_i$.

\begin{proof}[Proof of Theorem \ref{theorem-mme-simple}]

Consider $f_0\in \mathcal{V}_1\cup \mathcal{V}_2$. By \cite{MP1} and \cite{MPP}, $f_0$ has a finite number of ergodic measures of maximal entropy. Let $\upsilon_1,\dots,\upsilon_N$ be these measures.

By Lemma 3.1 of \cite{MP1} and Lemma 5.3 of \cite{MPP}, the measures $\upsilon_1,\dots, \upsilon_N$ are hyperbolic. Moreover, Katok's Horseshoe Theorem \cite{K}, implies that there exist horseshoes $\Lambda_j$ which are good approximations of the dynamics for every $1\leq j \leq N$. By Theorem 2.12 of \cite{BCS1}, we can assume that $\Lambda_j\sim_h \upsilon_j$. Observe that since $f\in \mathcal{V}_1 \cup \mathcal{V}_2$ is partially hyperbolic, then every $\upsilon_j$ is of saddle type. 

Applying Theorem 3.2 of \cite{ACW}, we can find $f$ $C^r$-arbitrarily close to $f_0$ such that for every $1\leq j\leq N$, the continuation $\Lambda_j(f)$ admits a dominated splitting into one dimensional bundles.

For $1\leq j\leq N$, consider $p_j,z_j\in \Lambda_{j}(f)$ such that $p_j$ is a hyperbolic periodic point and $z_j\in W^u_f(p_j)\pitchfork W^s_f(p_j)$. Denote $n(p_j)$ the period of $p_j$. Since $\Lambda_j(f)$ admits a dominated splitting into one dimensional bundles, we can assume, by a small perturbation if necessary, that for every $1\leq j\leq N$, all the eigenvalues of $Df^{n(p_j)}_{p_j}\vert_{E^{uu}_{p_j}}$ have distinct absolute values. 

Again by Lemma 3.1 of \cite{MP1} and by Lemma 5.4 of \cite{MPP}, there exist $\chi>0$ and $\mathcal{V}_0$, a $C^2$ neighborhood of $f_0$, such that for every $g\in \mathcal{V}_0$, and any ergodic measure of maximal entropy of $g$, $\upsilon$, then \begin{equation}\label{semi.cont.entr}
\begin{aligned}
h_{\upsilon}(g)&-h^u(g)>\chi, \; \text{ if } g\in \mathcal{V}_1 \text{ or },\\
h_{\upsilon}(g)&-\max\{h^u(g), h^s(g)\}>\chi,\;  \text{ if } g\in \mathcal{V}_2,
\end{aligned}
\end{equation} and $\upsilon$ is $\chi$-hyperbolic. 

Since $f$ is arbitrarily close to $f_0$, we can suppose that $f\in \mathcal{V}_0$. Moreover, by conditions \eqref{eq.condition1} and \eqref{eq.condition2}, we know that $\mathcal{F}=Df\vert_{E_f^{uu}}$ is a $\chi/2$-fiber bunched cocycle. See Equation (\ref{cu}) and Definition \ref{fb}. Then, the following two limits exist, \begin{equation}\label{holonomies.difeo} 
\begin{aligned}
H^{s, f}_{z_j,p_j}&=\lim_{n\to \infty} (\mathcal{F}^{n}_{p_j})^{-1}\circ \mathcal{F}^{n}_{z_j},\\
H^{u, f}_{p_j,z_j}&=\lim_{n\to \infty} (\mathcal{F}^{-n}_{z_j})^{-1}\circ \mathcal{F}^{-n}_{p_j}.
\end{aligned}
\end{equation}

Choose $\delta>0$ small enough such that for every $1\leq j\leq N$, the iterates of $p_j$ and $z_j$ do not belong to $\cup_{j=1}^N B_{\delta}(z_j)$. Let $g_0$ be a $C^r$ perturbation of $f$ supported on $\cup_{j=1}^N B_{\delta}(z_j)$ such that $g_0(z_j)=f(z_j)$ and $D{g_0}(z_j)=Df(z_j) \circ R_j$ where $R_j\colon T_{z_j}M\to T_{z_j}M$ is some linear isomorphism preserving $E^{uu}_{z_j}$.

Observe that $p_j,z_j \in \Lambda_{j}(g_0)$ and $z_j\in W^s_{g_0}(p_j)\pitchfork W^u_{g_0}(p_j)$. Moreover, \begin{equation}\label{transition} \psi^{g_0}_{p_j,z_j}:=H^{s, g_0}_{z_j,p_j}\circ H^{u, g_0}_{p_j,z_j}=H^{s, f}_{z_j,p_j}\circ R_j\vert E^{uu}_{z_j}\circ H^{u, f}_{p_j,z_j}.
\end{equation} Therefore, choosing the maps $R_j$ properly, we can conclude that for every $1\leq j\leq N$, the map $\psi^{g_0}_{p_j,z_j}$ does not preserves the eigenspaces associated to $E^{uu}_{p_j}$.

There exists a $C^2$ neighborhood of $g_0$, small enough, $\mathcal{W}$, such that for every $g\in \mathcal{W}$ and every $1\leq j\leq N$, there exist $p_{j}(g)$ and $z_{j}(g)$ verifying Definition \ref{pinch.twist.nuh}. Here $p_j(g)$ is the continuation of $p_j$ and we are using the continuity of the holonomies in Equation (\ref{holonomies.difeo}) as functions of the diffeomorphism and the points. 

By \cite[Corollary~5.3]{MP1}, in the case $f_0\in \mathcal{V}_1$, and \cite[Lemma~5.6]{MPP}, in the case $f_0\in \mathcal{V}_2$, if $g\in \mathcal{W}$ is close enough to $f_0$, for every ergodic measure of maximal entropy $\upsilon$ of $g$, there exists $0\leq j\leq N$ such that $p_j(g)\sim_h \upsilon$. Then, by the previous paragraph and Theorem~\ref{theorem-criterium}, we have that for every $g\in \mathcal{W}$, the cocycle $\mathcal{G}^u=Dg\vert_{E^{uu}_g}$ has simple Lyapunov spectrum for every measure of maximal entropy. Recall that Equation (\ref{semi.cont.entr}) guarantee that the measure $\upsilon$ is $\chi$-hyperbolic and the cocycle $\mathcal{G}^u$ is $\chi/2$-fiber bunched. 

Analogously, we can repeat the argument to obtain a subset of $\mathcal{W}$ where the cocycle $\mathcal{G}^s=Dg\vert_{E^{ss}_g}$ also has simple Lyapunov spectrum. This concludes the proof of the theorem for the case $\dim E^c=1$. Observe that it is also sufficient to conclude the case $\dim E^c=2$, since every $f\in \mathcal{V}_2$ admits a dominated splitting in $E^c$.
\end{proof}

\begin{proof}[Proof of Theorem~\ref{theo-intro-DA}]
   By \cite{FPS}, the unstable manifolds of $f$ have quasi isometric lift to the universal cover, then the same proof of \cite[Theorem B]{MPP}, see also \cite{MP2}, can be applied to show that $h^u(f)<h_{top}(f)$. Taking $\delta>0$ small enough, we know that $f$ is in $\mathcal{V}_1$, then we conclude the result.
\end{proof}

\section{Homoclinic classes of hyperbolic measures}

Let $m$ denote a probability measure in the Lebesgue class. We define $\mathit{Diff}^r_{m}(M)$ as the set of volume-preserving $C^r$ diffeomorphisms and denote $\mathit{PH}^r_{m}(M)$ the subset of $\mathit{Diff}^r_{m}(M)$ consisting of partially hyperbolic diffeomorphisms. 

Consider $f\in \mathit{PH}^1_{m}(M)$ such that $m$ is ergodic. The \emph{center Lyapunov exponents} of $f$ are the Lyapunov exponents of $Df\vert_{E^c}$ and we denote them by $\lambda^c_1(f)\geq \dots \geq \lambda_{d_c}(f)$, counted with multiplicity, where $d_c=\operatorname{dim}E^c$. We say that the center Lyapunov exponents of $f$ vary continuously if for every sequence $f_k\to f$ in $\mathit{Diff}^1_{m}(M)$ such that $f_k$ is ergodic, we have $\lambda^c_i(f_k)\to \lambda^c_i(f)$ for $i=1,\dots,d_c$.

A map $f\in \mathit{Diff}^2_{m}(M)$ is said to be \textit{non-uniformly hyperbolic} if $m$ is hyperbolic. That is, if for $=1,\dots,d_c$, $\lambda_i^c(f,x)$ is non-zero for $m$-almost every $x\in M$.

Let $\mathcal{A}\subset \mathit{PH}^2_{m}(M)$ be the set formed by diffeomorphisms such that they are non-uniformly hyperbolic, stably-ergodic and the center Lyapunov exponents vary continuously. By the results in \cite{LMY, M} this set has non-empty interior in $\mathit{Diff}^2_{m}(M)$ if $M=\mathbb{T}^{2d}$, $d\geq 2$.  See the proof of Corollary \ref{theorem-main}.


Observe that for every $f\in \mathcal{A}$, $m$ is an hyperbolic measure of saddle type. The proposition below is the main result of this section and its corollary is going to be a key element to conclude Theorem \ref{theo-intro}.

\begin{prop}\label{prop.homocl}
    Let $f\in \mathcal{A}$. If $p$ is a periodic hyperbolic point for $f$ such that $p \sim_h m$, then for any $C^2$ volume-preserving diffeormorphism $g$, $C^1$ sufficiently close to $f$, the hyperbolic continuation $p_g$ satisfies $p_g \sim_h m$. That is, $p_g$ is also homoclinically related to $m$.
\end{prop}

\begin{coro}\label{prop.homocl.ferradura}
    Let $f\in \mathcal{A}$. If $\Lambda$ is a hyperbolic transitive set which is homoclinically related to $m$, then for any $C^2$ volume-preserving diffeomorphism $g$, $C^1$ sufficiently close to $f$, the hyperbolic continuation $\Lambda_g$ is also homoclinically related to $m$.
\end{coro}
\begin{proof}
Take a periodic point $p\in \Lambda$, by Proposition~\ref{prop.homocl}, for any $C^2$ volume-preserving diffeomorphism $g$, $C^1$ sufficiently close to $f$, the hyperbolic continuation of $p$, $p_g$, is homoclinically related to $m$. As $p_g\in \Lambda_g$ this implies that $\Lambda_g\sim_h m$.
\end{proof}

In the following be present some definitions and results that we will need to prove Proposition \ref{prop.homocl}.

\begin{defn}\label{Pesin.block} Let $f$ be a $C^1$ diffeomorphism. Given $C>0$, $\chi>0$ and $\varepsilon>0$, we denote $\Delta_{C,\chi,\varepsilon}$ the Pesin block defined by the following property: for every $x\in \bigcup_{n\in \mathbb{Z}} f^n(\Delta_{C,\chi,\epsilon})$, there is a decomposition $T_xM=E_x\oplus F_x$ such that for every $y\in \Delta_{C,\chi,\varepsilon}$, $n\in \mathbb{Z}$ and $j\geq 0$, $$\max \left( \|Df^j\vert_{E_{f^n(y)}}\|, \|Df^{-j}\vert_{F_{f^{n}(y)}} \| \right)\leq C\, \exp(-\chi j + \varepsilon |n|).$$
\end{defn}

\begin{lemma}\label{prop.uniform.pesin}
    Let $f\in \mathcal{A}$ and $\chi>0 $ be such that $m$ is $\chi$-hyperbolic. If $\chi'=\chi/4$, then for every $\varepsilon>0$ and $\theta>0$, there exist a constant $C>0$ and a $C^1$-neighborhood of $f$, $\mathcal{U} = \mathcal{U}(\chi, \varepsilon, \theta)$,  such that for every $C^2$ volume-preserving diffeomorphism $g\in \mathcal{U}$, $m(\Delta_{C,\chi',\varepsilon}(g))> 1- \theta$.
\end{lemma}

This result is essentially contained in Sections 2 and 3 of \cite{BCS2}. For completeness,  let us explain the argument.  We will need the following version of the Pliss lemma.
\begin{lemma}[\cite{BCS2} Lemma $2.18$]\label{lem.pliss}
Let $T$ be an invertible measure preserving map on a probability space $(\Omega, \mathcal{B},\upsilon)$. Suppose that $\varphi \in L^{\infty}(\upsilon)$, $L\in \mathbb{R}$, $\kappa \in [0,1]$, and $\upsilon\{x\in \Omega: \varphi(x) >L\} \leq \kappa$. Then, for any $\beta <L$,
\[
\upsilon\{x\in \Omega:  \varphi_j(x) \geq \beta j\} \geq \frac{\int \varphi d\upsilon - \beta - \kappa \|\varphi - L\|_{\infty}}{L-\beta},
\]
where $\varphi_j(x) = \varphi(x) + \cdots+ \varphi(T^{j-1}(x))$. 
\end{lemma}

We will also use the proposition below.

\begin{prop}[\cite{BCS2} Proposition $2.21$]\label{prop.hyptimestopesinblock}
Let $f$ be a $C^1$ diffeomorphism. Given $N_0 \geq 1$ and $\chi >0$,  $P_{N_0} = P_{N_0}(\chi)$ be defined as 
\[
P_{N_0} := \left\{x: \begin{array}{l} \textrm{ there is a splitting $T_xM = E_x \oplus F_x$ s.t. $\forall j\geq 0$}\\ \|Df^{jN_0}(x)|_{E}\| \leq e^{-\chi j N_0} \textrm{ and } \|Df^{-jN_0}(x)|_F\| \leq e^{-\chi j N_0} \end{array} \right\}.
\]
Write $C = C(f, N_0):= \max_{k\in [1, N_0]} \max\{\|Df^k\|, \|Df^{-k}\|\}.$Then,  for every $\varepsilon >0$, the Pesin block $\Delta_{C,\chi, \varepsilon}$ verifies
\[
\upsilon(M - \Delta_{C,\chi, \varepsilon}) \leq 4 \max \left\{ \frac{c_0}{\varepsilon}, 1 \right\} \upsilon(M - P_{N_0}), \textrm{ for all invariant measures $\upsilon$},
\]
where $c_0 = \chi + \max\{\log\|Df\|, \log \|Df^{-1}\|\}.$
\end{prop}

\begin{proof}[Proof of Lemma \ref{prop.uniform.pesin}]
Let us do some estimates first for $f$.  We will then use the continuity of the Lyapunov exponents for ``spread'' these estimates for a neighborhood of $f$. 

Since $E^{ss}$ and $E^{uu}$ have uniform hyperbolicity, we only need to deal with the center direction. Suppose that $f$ has some positive and some negative center exponent. The other cases are analogous.

For an $m$-typical point, let  $E_x^-$ be the Oseledets direction in the center corresponding to the negative  Lyapunov exponents, and let $E_x^+$ be the direction corresponding to the positive exponents in the center.  Recall that $m$ is $\chi$-hyperbolic. For each $N\in \mathbb{N}$, define \[
K_{N} := \{ x\in M : \|Df^{N}_x|_{E_x^-}\| < e^{-\chi N} \textrm{ and } \|Df^{-N}_x|_{E_x^+}\| < e^{-\chi N}\}.
\]
The following claim is a direct consequence of the hyperbolicity of $m$.
\begin{claim}\label{claim.largemeasureK}
For every $\delta>0$, if $N$ is sufficiently large, then
\[
m(K_N) > 1-\delta.
\]
\end{claim}

\begin{claim}\label{claim.largemeasureP}
Let $\chi' = \frac{\chi}{4}. $ For every $\rho>0$, there exists $\delta>0$,  such that if $m(K_N) >1-\delta$, for some $N\in \mathbb{N}$,  then $m(P_N(\chi')) > 1-\rho$. 
\end{claim}
\begin{proof}
Suppose that $K_N$ has measure larger than $1-\delta$. We will see later what is the dependence of the measure of $P_N$ on $\delta$.   Consider the function $\psi(x) = \chi_{K_N}(x)$, where $\chi_{K_N}$ is the characteristic function.  Let $\psi^+$ be the forward Birkhoff average of $\psi$.  By ergodicity of $m$, for $m$-almost every $x$, $\psi^+(x) = \int \psi dm =  m(K_N) >1-\delta$.  

Consider $H = \max \{\|Df\|, \|Df^{-1}\|\}$ and take $\beta = \max\{ 1- \frac{\chi}{2 \log H}, \frac{3}{4}\}$.  Take $L=1$ and $\kappa =0$.  By Lemma \ref{lem.pliss},  we obtain that 
\[
m\left\{x: \psi_j(x) \geq \beta j, \textrm{ for all }j\geq 0\right\} \geq \displaystyle \frac{\int \psi dm - \beta}{1-\beta} \geq \frac{1-\delta - \beta}{1-\beta} = 1- \frac{\delta}{1-\beta}.
\]
Write $K^+_N:=\left\{x:\psi_j(x) \geq \beta j, \textrm{ for all }j\geq 0\right\} $.  For $x\in M$,  $j\in \mathbb{N}$, set $G(x,j,N):= \{i\in [0, j-1]: f^{iN}(x) \in K_N\}$. If $x\in K^+_N$, we have that for every $j\in \mathbb{N}$, 
\[
\displaystyle \frac{\# G(x,j,N)}{j} \geq \beta.
\]
For such $x$, we obtain
\[
\begin{array}{ll}
&\|Df^{jN}_x|_{E^-}\|\leq \\
&\displaystyle \left(\prod_{i\in G(x,j,N)} \|Df^N_{f^{iN}(x)}|_{E^-}\|\right)\left( \prod_{i\in (G(x,j,N))^c \cap [0,j-1]} \|Df^N_{f^{iN}(x)}|_{E^-}\|\right)\\
& = I . II
\end{array}
\]
From the definition of $K_N$, we have
\[
I \leq e^{-\chi N \# G(x,j,N)} \leq e^{-\chi  N \beta j},
\]
and 
\[
II \leq e^{N \log H (j - \# G(x,j,N))}  \leq e^{N \log H j(1-\beta)}.
\]
From the definition of $\beta$, we obtain 
\[
-\chi  N \beta j + N \log H j (1-\beta)  \leq -\frac{3 \chi }{4} Nj +  \frac{\chi}{2}Nj = -\frac{\chi}{4} Nj.
\]
Therefore, 
\[
\|Df^{jN}_x|_{E^-}\| \leq e^{-\frac{\chi }{4} jN},
\]
for every $j\in \mathbb{N}$, whenever $x\in K^+_N$.  Define similarly $K^-_N$ for past iterates and $E^+$ instead of $E^-$. By the same computation,  we conclude that 
\[
m(K^-_N) >1 - \frac{\delta}{1-\beta}.
\]
Thus,
\[
m(K^-_N \cap K^+_N )  = 1- m((K^-_N)^c \cup (K^+_N)^c) \geq 1 - \frac{2\delta}{1-\beta}.
\]
Given $\rho>0$, take $\delta = \frac{\rho(1-\beta)}{2}$ and take $N$ large enough such that $m(K_N) > 1- \delta$. Thus, $m(P_N(\chi')) > 1-\rho$. 
\end{proof}

Just as above, set $\chi' = \frac{\chi}{4}$. Fix $\varepsilon>0$ and $\theta>0$.  Set $L' = 4\max\left\{\frac{c_0}{\varepsilon}, 1\right\}$, where $c_0 = \chi' + \max\{\log \|Df\|, \log \|Df^{-1}\|\}$. Take $\rho < \frac{\theta}{L'}$. By Claims \ref{claim.largemeasureK} and \ref{claim.largemeasureP}, we fix $\delta>0$ small, and $N$ large such that $m(K_N) > 1-\delta$, which implies that $m(P_N(\chi')) > 1-\rho$.  By Proposition \ref{prop.hyptimestopesinblock}, 
\[
m(\Delta_{C,\chi', \varepsilon}) > 1-\theta,
\]
where $C = C(f,N)$ is given in the statement of that proposition.  Observe that $H$ in the proof of Claim \ref{claim.largemeasureP} and $c_0$ and $C$ defined above are uniform in a $C^1$-neighborhood of $f$.  Also, the starting point for obtaining this estimate was that $m(K_N)> 1-\delta$. Consider $\widetilde{K}_N$ to be the set contained in the Grassmanian of $E^c$ defined by points of the form $(x,E_x^*)$, where $x\in K_N$ and $*=-$ or $+$.  Observe that, the condition defining $K_N$ is open, meaning, there are two open sets $U^s$ and $U^u$, 
with the property that if $(x,E) \in U^s$, then $\|Dg^N_x|_{E}\| \leq e^{-\chi' N}$, for every $g$ sufficiently $C^1$-close to $f$. Similar conclusion for $U^u$ and past iterates. 

Since $f$ is stably-ergodic, every $g\in\mathit{Diff}^2_{m}(M)$ which is $C^1$-sufficiently close to $f$ is ergodic. Moreover, by the continuity of the center Lyapunov exponents, $m$ is hyperbolic of saddle type for $g$. Let $m^+_g = \int \delta_{E^+_g(x)} dm$ and $m^-_g= \int \delta_{E^-_g(x)} dm$ be the unstable and stable lifts of $m$ for $g$ in the Grassmanian of $E^c$, respectively. Again by the continuity of the center Lyapunov exponents if $g_k\to f$ in $\mathit{Diff}^1_{m}(M)$, then for $*\in \{+, -\}$, $m^*_{g_k} \to m^*$ in the weak*-topology, see \cite{BP}.  Therefore, for every $g$ $C^r$ map in a $C^1$ neighborhood of $f$, $m^+_g(U^u) > 1-\delta$ and $m^-_g(U^s) >1-\delta$, which implies that $m(K_N(g)) > 1-\delta$.  The rest of the estimates holds in a uniform way, since $N$, $\varepsilon$ and $\theta$ are fixed. Therefore, $m(\Delta_{C,\chi', \varepsilon}(g))>1-\theta$. 
\end{proof}

For $*\in \{s,u\}$ and $\delta>0$, we denote $W^{*}_{\delta}(x)=W^{*}(x)\cap B^{*}(x,\delta)$.

\begin{lemma}\label{prop.big.transverse}
    Let $f\in \mathcal{A}$. If $p$ is a periodic hyperbolic point for $f$ such that $p \sim_h m$, then for every $\delta>0$ there exists $L>0$ and a compact set $M_\delta\subset M$, $m(M_\delta)>1-\delta$ such that $W^{s}_L(\mathcal{O}(p))\pitchfork W^u_\delta(x)\neq\emptyset$ for every $x\in M_\delta$.
\end{lemma}
\begin{proof}
    Recall Definition \ref{homoclinic.point}. Up to reducing $\delta$, by Pesin theory \cite{P}, we can find a compact set $M_\delta$, $m(M_\delta)>1-\delta$, such that for every $x\in M_\delta$, the Pesin unstable manifold, $W^u(x)$, contains a disk of size $\delta$, $W^s(\mathcal{O}(p))\pitchfork W^u(x)\neq\emptyset$, $x$ is a recurrent point and the map $M_\delta \ni x\mapsto W^s_\delta(x)$ is continuous. Using the Inclination Lemma of \cite{BCS1} we can assume that $W^s(\mathcal{O}(p))\pitchfork W^u_\delta(x)\neq\emptyset$.

    By compactness of $M_\delta$ and continuity of $x\mapsto W^s_\delta(x)$ there exists $L>0$ such that $W^s_L(\mathcal{O}(p))\pitchfork W^u_\delta(x)$.
\end{proof}

\begin{proof}[Proof of Proposition~\ref{prop.homocl}]
    Suppose there exists $f_k\to f$ such that for every $k$ the conclusion does not hold. Fix $\varepsilon>0$ and $\theta>0$ and let $\Delta_k=\Delta_{C,\chi',\varepsilon}(f_k)$ be given by Lemma~\ref{prop.uniform.pesin}. Up to taking a subsequence we can assume that $\Delta_k\to \Delta$ in the Hausdorff topology of compact sets. Observe that the complement of $\Delta_{C,\chi',\varepsilon}(g)$ is open in $(x,g)$, then $\Delta\subset \Delta_{C,\chi',\varepsilon}(f)$.

    \begin{claim}
        $m(\Delta)\geq 1-\theta$.
    \end{claim}
\begin{proof}
    Take $\beta>0$, for $k$ large $\Delta_k\subset B_\beta(\Delta)$, then $m(B_\beta(\Delta))>1-\theta$, for every $\beta>0$. As $\Delta$ is compact, $m(\Delta)\geq 1-\theta$.
\end{proof}
    
\begin{rem}\label{rem.cont.manifold} By the uniformity of the parameters in $\Delta_k$, we can choose $\delta>0$ such that $W^u(f_k,x_k)$ has diameter bigger than $\delta$ for every $x_k\in\Delta_k$ if $k$ is big enough. Moreover, if $x_k\in \Delta_k$ and $x_k\to x$ with $x\in \Delta$, then $W^u_\delta(f_k, x_k)\to W^u_\delta(x)$ uniformly. This is a consequence of the graph transform used to prove the existence of $W^u$, see \cite{P}.
\end{rem}

Take $\delta<\theta$ such that the unstable manifolds in $\Delta_k$ have diameter bigger than $\delta$ and let $M_\delta$ be given by Lemma~\ref{prop.big.transverse}. 
    
    Let $p_k$ be the $f_k$ hyperbolic periodic point that is the continuation of $p$. By continuity of the unstable manifold with respect to $(p_k,f_k)$ we have that $W^s_L(f_k,\mathcal{O}(p_k))\pitchfork W^u_\delta(f,x)\neq\emptyset$ for every $x\in M_\delta$.

    Observe that $m(\Delta_k\cap M_\delta)>0$. Moreover, if $x_k\in \Delta_k\cap M_\delta$ and $x_k\to x$ with $x\in \Delta$, then by Remark \ref{rem.cont.manifold}, $W^u_\delta(f_k, x_k)\to W^u_\delta(x)$. Therefore, for $k$ big enough, we conclude that $W^s_L(f_k,\mathcal{O}(p_k))\pitchfork W^u_\delta(f_k,x_k)\neq\emptyset$. 

    The same calculation, changing the stable manifolds by the unstable manifolds,  shows that $p_k$ is homoclinically related to $\mu$. This contradiction proves the proposition.  \qedhere
\end{proof}

\section{Proof of Theorem~\ref{theo-intro} and Corollaries}\label{section.proofA}

Recall that for fixed $r\geq 2$, $\mathcal{A}\subset \mathit{PH}^r_{m}(M)$ denotes the set formed by diffeomorphisms such that they are non-uniformly hyperbolic, stably-ergodic and the Lyapunov exponents of $Df\vert_{E^c}$ vary continuously. 

If $f$ is $\chi$-hyperbolic, then the $u$-bunched condition, see Equation (\ref{cu}), implies that there exists $C>0$ and $\lambda\in (0,1)$ such that
$$\|Df^n|_{E^{uu}}\|\|(Df^n|_{E^{uu}})^{-1}\|e^\frac{-\chi \alpha |n|}{2}<C\lambda^{|n|}.$$
for $n\in \mathbb{Z}$. In this case, the cocycle $\mathcal{F}=Df\vert_{E^{uu}}$ is $\chi/2$-fiber bunched according to Definition \ref{fb}. 

Observe that for every $f\in \mathcal{A}$ which is $\chi$-hyperbolic and $u$-bunched, there exists a $C^1$ neighborhood of $f$ such that any map in this neighborhood is $\chi$-hyperbolic and $u$-bunched.

\begin{theo}\label{theorem-u-simple}
    Fix $r\geq 2$. Let $f\in \mathcal{A}$ be $u$-bunched. Then, $f$ can be $C^r$-approximated by $C^2$ open sets of volume-preserving maps with simple unstable Lyapunov spectrum. That is, the Lyapunov exponents associated to $Dg\vert_{E^{uu}}$ have multiplicity one.
\end{theo}
\begin{proof}
    Consider $f_0\in \mathcal{A}$ being $u$-bunched. Since $f_0$ is non-uniformly hyperbolic, by Katok's Horseshoe Theorem \cite{K} and Theorem 2.12 of \cite{BCS1}, we know there exists a horseshoe $\Lambda$ such that $\Lambda\sim_h m$. 

    Let $\mathcal{V}_0$ denote a neighborhood of $f_0$ in $\mathit{Diff}^2_{m}(M)$. As before, by Theorem 3.2 of \cite{ACW}, there exists $f\in \mathcal{V}_0$ such that the continuation $\Lambda_{f}$ admits a dominated splitting into one dimensional bundles.

    Since $\mathcal{V}$ is an open set, $f\in \mathcal{V}$ and we can suppose $\mathcal{V}_0$ is small enough to apply Proposition \ref{prop.homocl.ferradura}. Therefore, $\Lambda_{f}\sim_h m$.

    Using the same arguments than in the proof of Theorem \ref{theorem-mme-simple}, we can find an open set $\mathcal{W}\subset \mathit{Diff}^2_{m}(M)$, arbitrarily $C^r$ close to $f$, such that for every $g\in \mathcal{W}$, there exist $p(g)$ and $z(g)$ verifying Definition \ref{pinch.twist.nuh}. Moreover, $p(g)\sim_h m$. Here we are using again Proposition \ref{prop.homocl.ferradura}. Applying Theorem \ref{theorem-criterium} to the cocycle $\mathcal{G}=Dg\vert_{E^{uu}_g}$ we conclude the proof. 
\end{proof}

\begin{proof}[Proof of Theorem \ref{theo-intro}]
    We can apply Theorem \ref{theorem-u-simple} to conclude that there exists $\mathcal{W}$ arbitrarily close to $f$ such that every $g\in \mathcal{W}$ has simple unstable Lyapunov spectrum. Analogously, considering $f^{-1}$, we prove that the stable Lyapunov spectrum is also simple.
\end{proof}

Let $f$ be a partially hyperbolic diffeomorphism. We say that $f$ is center bunched if the functions in Equation (\ref{ph}) satisfy, $$\nu< \gamma \widehat{\gamma} \qquad \text{and} \qquad \widehat{\nu}< \gamma \widehat{\gamma}.$$ Observe that this condition is $C^1$-open and every partially hyperbolic map with $\dim E^c=1$ is center-bunched.  

We say that $f$ is \textit{accessible} if for any two points $x,y\in M$, there exists a path that connects $x$ to $y$, which is a concatenation of finitely many subpaths, each of which lies entirely in a single leaf of $W^{ss}$ or a single leaf of $W^{uu}$. 

By \cite{BW}, every $C^2$ volume-preserving partially hyperbolic map which is center-bunched and accessible is ergodic. 

\begin{proof}[Proof of Corollary \ref{cor-intro}]
    As $\int\log \|Df|E^c\|dm$ varies continuously with $f$, there exist $\chi>0$ and a $C^1$ open set $\mathcal{V}$, such that $|\int\log \|Dg|E^c\|dm|>\chi>0$, for every $g\in \mathcal{V}$. If $\delta>0$ is small enough and $c(f)<\delta$, then any $g\in \mathcal{V}$ is $su$-bunched for this $\chi$. By \cite{RHRHU}, up to a perturbation we can find $g\in \mathcal{V}$ such that is $C^r$-arbitrarily close to $f$ and $g$ is stably ergodic. The conclusion follows applying Theorem~\ref{theo-intro} to $g\in \mathcal{A}$. 
\end{proof}

\begin{proof}[Proof of Corollary \ref{theorem-main}] Let $f:M\to M$ be a $C^r$ partially hyperbolic symplectic diffeomorphism which is center bunched, accessible and with $2$-dimensional center bundle. Suppose also that $f$ admits a periodic point. Theorem A of \cite{SW} implies the existence of maps in $\mathit{PH}^r_{m}(\mathbb{T}^{2d})$ satisfying these hypotheses. 

Since the center bundle of $f$ is two-dimensional, \cite{AV2} implies that the accessibility property is $C^1$-open. 
Therefore, by \cite{BW} and the center bunching hypothesis, tehre exists $\mathcal{W}$, a $C^1$ neighborhood of $f$, such that every $C^2$ diffeomorphism in $\mathcal{W}$ is stably ergodic. 

By \cite{LMY, M}, there exists a non-empty set $\mathcal{V}\subset \mathcal{W}$ which is $C^2$-open such that $\mathcal{V}$ is $C^r$-arbitrarily close to $f$ and every $g\in \mathcal{V}$ is non-uniformly hyperbolic and a continuity point for the center Lyapunov exponents. Moreover, $\lambda^c_1(g)\neq \lambda^c_2(g)$ for every $g\in \mathcal{V}$.

By the conclusions above, there exists a non-empty set $\mathcal{V}\subset \mathcal{A}$. The corollary follows by Theorem \ref{theo-intro}.
\end{proof}

\begin{proof}[Proof of Corollary \ref{standard}]
Observe that by the definition of the standard map $g_{\lambda}$, the map $f\times g_0$ is a partially hyperbolic symplectic map which is center-bunched. Since these conditions are $C^1$-open, we conclude that there exists a neighborhood of $f\times g_{0}$ among symplectic skew-products over $f$ where these hypotheses hold. We denote this set $\mathcal{U}_1$.

By Theorem 4 of \cite{HS}, every map in $\mathcal{U}_1$ can be $C^r$-approxi\-mat\-ed by an open set $\mathcal{U}_2\subset \mathcal{U}_1$ such that every map in $\mathcal{U}_2$ is symplectic, accessible and has a periodic point. See Lemma 7.2 in \cite{HS}.

Since the perturbation of \cite{LMY, M} can be achieved among skew-products over $f$, we can conclude that every map in $\mathcal{U}_2$ can be $C^r$-approximated by a volume-preserving skew-product over $f$ which is non-uniformly hyperbolic and a continuity point for the center Lyapunov exponents. Moreover, this map has simple center Lyapunov exponents.

As a conclusion, we get that every map in $\mathcal{U}_2$ can be $C^r$-approximated by maps in $\mathcal{A}$ which are skew-products over $f$.

It is easy to see that if $g$ is a partially hyperbolic skew product over $f$ then $c(f)=c(g)$.
Corollary \ref{standard} follows by Theorem \ref{theo-intro}, since $c(f)=0$ and $f\times g_{\lambda}\in \mathcal{U}_1$ for every $\lambda$ close enough to zero.
\end{proof}

{\em{Acknowledgements.}} We are thankful to Yuri Lima and Sylvain Crovisier for helpful conversations on this work. K.M. thanks the hospitality of Universidade Federal do Cear\'a during her postdoc when this work was elaborated.


\begin{thebibliography}{99}

\bibitem [ACW]{ACW}
A. Avila, S. Crovisier, A. Wilkinson.
$C^1$ density of stable ergodicity.
\emph{Adv. Math.} \textbf{379} (2021), 68 pp.


\bibitem[ASV]{ASV} A. Avila, J. Santamaria, M. Viana.
Holonomy invariance: rough regularity and applications to Lyapunov exponents. 
\emph{Astérisque} \textbf{358} (2013), 13-74.

\bibitem [AV]{AV}
A. Avila, M. Viana.
Simplicity of Lyapunov spectra: a sufficient criterion.
\emph{Port. Math.} \textbf{64} (2007), 311-376.

\bibitem[AV1]{AV1}
A. Avila and M. Viana.
Simplicity of Lyapunov spectra: proof of the Zorich–Kontsevich conjecture.
\emph{Acta Math.} \textbf{198} (2007), no. 1, 1-56.

\bibitem [AV2]{AV2}
A. Avila, M. Viana.
Stable accessibility with 2-dimensional center.
\emph{Astérisque} \textbf{416} (2020), 301-320.


\bibitem [BP1]{BP}
L. Backes, M. Poletti.
Continuity of Lyapunov exponents is equivalent to continuity of Oseledets subspaces.
\emph{Stoch. Dyn.} \textbf{17} (2017), no. 6.

\bibitem [BPVL]{BPV}
L. Backes, M. Poletti, P. Varandas, Y. Lima.
Simplicity of Lyapunov spectrum for linear cocycles over non-uniformly hyperbolic systems.
\emph{Ergod. Th. Dynam. Sys.} \textbf{40} (2020), no. 11, 2947-2969.

\bibitem [BP2]{BP2}
J. Bezerra, M. Poletti.
Random product of quasi-periodic cocycles.
\emph{Proc. Amer. Math. Soc.} \textbf{149} (2021), no. 7, 2927-2942.

\bibitem [BDV]{BDV}
C. Bonatti, L. J. Díaz, M. Viana.
\emph{Dynamics beyond uniform hyperbolicity.}
Encyclopedia of Mathematical Sciences, vol. 102.
Springer-Verlag, 2005.

\bibitem [BGV]{BGV}
C. Bonatti, X. Gómez-Mont, M. Viana.
Généricité d'exposants de Lyapunov non-nuls pour des produits déterministes de matrices.
\emph{Ann. Inst. H. Poincaré C, Anal. Non Linéaire} \textbf{20} (2003), no. 4, 579-624.

\bibitem [BV2]{BV2}
C. Bonatti, M. Viana.
Lyapunov exponents with multiplicity 1 for deterministic products of matrices.
\emph{Ergod. Th. Dynam. Sys.} \textbf{24} (2004), 1295-1330.

\bibitem [BW]{BW}
K. Burns, A. Wilkinson.
On the ergodicity of partially hyperbolic systems.
\emph{Ann. of Math.} (2) \textbf{171} (2010), no. 1, 451-489.

\bibitem [BCS1]{BCS1}
J. Buzzi, S. Crovisier, O. Sarig.
Measures of maximal entropy for surface diffeomorphisms.
\emph{Ann. of Math. (2)} \textbf{195} (2022), 421-508.

\bibitem [BCS2]{BCS2}
J. Buzzi, S. Crovisier, O. Sarig.
Strong positive recurrence and exponential mixing for diffeomorphisms.
Preprint, \emph{arXiv:2501.07455}, 2025.

\bibitem [BS]{BS}
J. Buzzi, O. Sarig.
Uniqueness of equilibrium measures for countable Markov shifts and multidimensional piecewise expanding maps.
\emph{Ergod. Th. Dynam. Sys.} \textbf{23} (2003), no. 5, 1383-1400.

\bibitem[FPS]{FPS}
T.~Fisher, R.~Potrie, and M.~Sambarino.
\newblock Dynamical coherence of partially hyperbolic diffeomorphisms of tori isotopic to anosov.
\newblock {\em Math. Z.} \textbf{278} (2014), no. 1, 149-168.

\bibitem [GM]{GM}
I. Ya. Gol’dsheid, G. A. Margulis.
Lyapunov indices of a product of random matrices.
\emph{Usp. Mat. Nauk.} \textbf{44} (1989), 13-60.

\bibitem [GR]{GR}
Y. Guivarc’h, A. Raugi.
Products of random matrices: convergence theorems.
\emph{Contemp. Math.} \textbf{50} (1986), 31-54.

\bibitem [HPS]{HPS}
M. Hirsch, C. Pugh, M. Shub.
Invariant Manifolds.
\emph{Lecture Notes in Math.} \textbf{583}, Springer-Verlag, 1977.

\bibitem[HS]{HS}
V. Horita and M. Sambarino. 
Stable ergodicity and accessibility for certain partially hyperbolic diffeomorphisms with bidimensional center leaves. \emph{Comment. Math. Helv.} \textbf{92} (2017), 467-512.

\bibitem [K]{K}
A. Katok.
Lyapunov exponents, entropy and periodic orbits for diffeomorphisms.
\emph{Inst. Hautes Études Sci. Publ. Math.} \textbf{51} (1980), 131-173.

\bibitem [LMY]{LMY}
C. Liang, K. Marin, J. Yang. 
Lyapunov exponents of partially hyperbolic volume-preserving maps with 2-dimensional center bundle.
\emph{Ann. Inst. H. Poincaré - Anal. Non Linéaire} \textbf{35} (2018), no. 6, 1687-1706.

\bibitem [LOP]{LOP}
Y. Lima, D. Obata, M. Poletti.
Measures of maximal entropy for non-uniformly hyperbolic maps.
Preprint, \emph{arXiv:2405.04676}, 2024.

\bibitem [M]{M}
K. Marin.
$C^r$-density of (non-uniform) hyperbolicity in partially hyperbolic symplectic diffeomorphisms.
\emph{Comment. Math. Helv.} \textbf{91} (2016), no. 2, 357-396.

\bibitem [MP1]{MP1}
J. C. Mongez, M. J. Pacifico.
Finite measures of maximal entropy for an open set of partially hyperbolic diffeomorphisms.
\emph{Trans. Amer. Math. Soc.} \textbf{377} (2024), 8695-8720.


\bibitem[MP2]{MP2}
J. C. Mongez, M. J. Pacifico.
\newblock Robustness and uniqueness of equilibrium states for certain partially
  hyperbolic systems.
\newblock {\em Nonlinearity}, 38(3):035020, 2025.

\bibitem [MPP]{MPP}
J. C. Mongez, M. J. Pacifico, M. Poletti.
Partially hyperbolic diffeomorphisms with a finite number of measures of maximal entropy.
Preprint, \emph{arXiv:2502.17385}, 2025.

\bibitem [O]{O}
S. Ben Ovadia.
Symbolic dynamics for non uniformly hyperbolic diffeomorphisms of compact smooth manifolds. 
\emph{J. Mod. Dyn.} \textbf{13} (2018), 43-113.

\bibitem [P]{P}
J. B. Pesin. 
Families of invariant manifolds that correspond to nonzero characteristic exponents. 
\emph{Izv. Akad. Nauk SSSR Ser. Mat.} \textbf{40} (1976), 1332-1379, 1440.

\bibitem [PV]{PV}
M. Poletti, M. Viana.
Simple Lyapunov spectrum for certain linear cocycles over partially hyperbolic maps.
\emph{Nonlinearity} \textbf{32} (2019), no. 1, 238-284.

\bibitem[RHRHU]{RHRHU}
F. Rodríguez Hertz, M. A. Rodríguez Hertz, R. Ures. Accessibility and stable ergodicity for parially hyperbolic diffeomorphisms with 1D-center bundle.
\emph{Invent. Math.} \textbf{172} (2008), 353–381. 


\bibitem [S]{S}
O. Sarig.
Symbolic dynamics for surface diffeomorphisms with positive entropy.
\emph{J. Amer. Math. Soc.} \textbf{26} (2013), 341-426.

\bibitem [SH]{SH}
M. Shub.
\emph{Global stability of dynamical systems.}
Springer-Verlag, 1987.

\bibitem [SW]{SW}
M. Shub, A. Wilkinson.
Stably ergodic approximation: two examples.
\emph{Ergod. Th. Dynam. Sys.} \textbf{20} (2000), no. 3, 875-893.


\end{thebibliography}
\end{document}